\documentclass{amsart}
\usepackage{mathrsfs,amssymb,setspace}
\usepackage{amsmath}
\usepackage{verbatim, enumerate, color, stmaryrd, bbold, bm,dsfont,mathabx}
\usepackage{centernot}
\usepackage{tikz}
\usepackage{csquotes}
\usepackage[colorlinks=true,allcolors=black]{hyperref}
\usepackage[nameinlink]{cleveref}

\theoremstyle{plain}
\newtheorem{theorem}{Theorem}[section]
\newtheorem{lemma}[theorem]{Lemma}
\newtheorem{proposition}[theorem]{Proposition}
\newtheorem{corollary}[theorem]{Corollary}

\theoremstyle{definition}

\theoremstyle{remark}
\newtheorem{remark}[theorem]{Remark}

\newcommand\numberthis{\addtocounter{equation}{1}\tag{\theequation}}
\begin{document}


\title[Power Graphs]{On the minimum degree, edge-connectivity and connectivity of power graphs of finite groups}
\author[Ramesh Prasad Panda and K. V. Krishna]{Ramesh Prasad Panda and K. V. Krishna}
\address{Department of Mathematics, Indian Institute of Technology Guwahati, Guwahati, India}
\email{\{r.panda, kvk\}@iitg.ac.in}


\begin{abstract}
The power graph of a group $G$ is the graph whose vertex set is $G$ and two distinct vertices are adjacent if one is a power of the other. In this paper, the minimum degree of power graphs of certain classes of cyclic groups, abelian $p$-groups, dihedral groups and dicyclic groups are obtained. It is ascertained that the edge-connectivity and minimum degree of power graphs are equal, and consequently the minimum disconnecting sets of power graphs of the aforementioned groups are determined. Then the equality of connectivity and minimum degree of power graphs of finite groups is investigated and in this connection, certain necessary conditions are produced. A necessary and sufficient condition for the equality of connectivity and minimum degree of power graphs of finite cyclic groups is obtained. Moreover, the equality is examined for the power graphs of abelian $p$-groups, dihedral groups and dicyclic groups.
\end{abstract}

\subjclass[2010]{05C25, 05C40, 05C07, 20K01}

\keywords{Finite group, Power graph, Minimum degree, Connectivity}

\maketitle

\tableofcontents

\section{Introduction}

The notion of a directed power graph was introduced by Kelarev and Quinn \cite{kelarev2000combinatorial, kelarevDirectedSemigr}. The \emph{directed power graph} $\overrightarrow{\mathcal{G}}(S)$ of a semigroup $S$ is a directed graph with vertex set $S$ and there is an arc from a vertex $u$ to another vertex $v$ if $v=u^\alpha$ for some natural number $\alpha \in \mathbb{N}$. Subsequently, Chakrabarty et al. \cite{GhoshSensemigroups} defined (\emph{undirected}) \emph{power graph} $\mathcal{G}(S)$ of a semigroup $S$ as the (undirected) graph with vertex set $S$ and distinct vertices $u$ and $v$ are adjacent if $v=u^\alpha$ for some $\alpha \in \mathbb{N}$ or $u=v^\beta$ for some $\beta \in \mathbb{N}$.

Since its introduction, power graphs has been a topic of interest for many researchers. Studies in the literature are not only to understand properties of power graphs of (semi)groups, but also to characterize (semi)groups through their power graphs. Cameron and Ghosh \cite{Ghosh} proved that two finite abelian groups with isomorphic power graphs are isomorphic. Cameron \cite{Cameron} showed that given two finite groups, if their power graphs are isomorphic, then their directed power graphs are also isomorphic. It was shown by Curtin and Pourgholi \cite{curtin2014edge, curtin2016euler} that among all finite
groups of a given order, the cyclic group of that order has the maximum number of edges and has the largest clique in its power graph. In \cite{MR3200118}, Moghaddamfar et al. supplied necessary and sufficient conditions for a proper power graph a group $G$ -- the graph obtained by removing identity element from power graph of $G$ -- to be a strongly regular graph, a bipartite graph or a planar graph. In \cite{bubbo17,doostabadi2015connectivity}, the authors studied the components of proper power graphs. Many other researchers have investigated different aspects of powers graphs, e.g. \cite{MR3266285, MR3514980, MR3612206}.

The vertex connectivity (or simply connectivity) of power graphs has been explored by various authors.
In \cite{ChattopadhyayConnectivity,chattopadhyay2015laplacian}, Chattopadhyay and Panigrahi found the connectivity for power graphs of certain groups and gave upper bounds of connectivity $\kappa(\mathcal{G}(\mathbb{Z}_n))$ for $n$ with two prime factors or product of three primes, where $\mathbb{Z}_n$ is the additive group of integers modulo $n$. Recently, Panda and Krishna \cite{ConPowerGr17} improved those results by supplying upper bounds of $\kappa(\mathcal{G}(\mathbb{Z}_n))$ for all $n$ and found its value for some $n$. In a more general setup, it is a result due to Whitney \cite{whitney1932congruent} that  $\kappa(\Gamma) \leq \kappa'(\Gamma) \leq \delta(\Gamma)$, where $\Gamma$ is a finite simple graph, and $\kappa(\Gamma)$, $\kappa'(\Gamma)$ and $\delta(\Gamma)$ are its connectivity, edge connectivity and the minimum degree, respectively. Therefore, it is a natural to ask when at least two of the above graph parameters are equal. Several results on this can be found in literature. Chartrand \cite{chartrand1966graph} proved that if $\delta(\Gamma) \geq \left[\displaystyle \frac{|V(\Gamma)|}{2} \right]$, then $\kappa'(\Gamma) = \delta(\Gamma)$. A sufficient condition for $\kappa'(\Gamma) = \delta(\Gamma)$ due to Plesnik \cite{plesnik1975critical} is that $\text{diam}(\Gamma) \leq 2$. In \cite{CriticallyChartrand}, Chartrand et al. showed that if $\Gamma$ is a graph with $\kappa(\Gamma)=n$ and $\kappa(\Gamma-v)=n-1$ for every $v \in V(\Gamma)$, then $\delta(\Gamma) < \displaystyle \frac{3n-1}{2}$. Moreover, if $\kappa(\Gamma)=n$ and $\kappa(\Gamma-e)=n-1$ for every $e \in E(\Gamma)$, then $\delta(\Gamma)=n$ \cite{Halin-n-conn}.  In this paper, we aim to study $\delta(\Gamma)$ when $\Gamma$ is a power graph of a finite group and investigate its relation with $\kappa(\Gamma)$ and $\kappa'(\Gamma)$.

 The paper is organized as follows. We begin with presenting some necessary background material in \Cref{s-prelim}. In \Cref{s-eq-md-ec}, we first observe that the edge-connectivity and minimum degree of power graphs of finite groups are equal. Then we focus on investigating the minimum degree of power graphs of cyclic groups in \Cref{s-cyclic}. In fact, we obtain the minimum degree of $\mathcal{G}(\mathbb{Z}_n)$ when $n$ has two prime factors or $n$ is a product of at most four distinct primes. Consequently, we give two sharp upper bounds of the minimum degree of $\mathcal{G}(\mathbb{Z}_n)$ for all $n$. Further, in \Cref{s-grps}, we determine the minimum degree of power graphs of abelian $p$-groups, dihedral groups and dicyclic groups. Along with minimum degree, we also obtain minimum disconnecting sets of power graphs of all these groups. In \Cref{s-equal}, we show that if the connectivity and minimum degree of power graph of a finite group are equal then the group is of even order and find the minimum degree (hence connectivity) when the group is cyclic. Finally, we supply a necessary and sufficient condition for the equality of the connectivity and minimum degree of power graphs of cyclic groups and address the same for abelian $p$-groups, dihedral groups and dicyclic groups.

\section{Preliminaries}
\label{s-prelim}

The set of vertices and the set of edges of a graph $\Gamma$ are always denoted by $V(\Gamma)$ and $E(\Gamma)$, respectively. A graph with no loops or parallel edges is called a \emph{simple graph}.  A graph with one vertex and no edges is called a \emph{trivial graph}. Given a graph $\Gamma$, the degree of a vertex $v$ is denoted by $\deg_{\Gamma}(v)$ or simply $\deg(v)$, and the minimum degree of $\Gamma$ is denoted by $\delta(\Gamma)$.

A \emph{separating set} of a graph $\Gamma$ is a set of vertices whose removal increases the number of components of $\Gamma$. A separating set is \emph{minimal} if none of its proper subsets separates $\Gamma$. A separating set of $\Gamma$ with least cardinality is called a \emph{minimum separating set} of $\Gamma$. A \emph{disconnecting set} of $\Gamma$ is a set of edges whose removal increases the number of components of $\Gamma$. A disconnecting set is \emph{minimal} if none of its proper subsets disconnects $\Gamma$. A disconnecting set of $\Gamma$ with least cardinality is called a \emph{minimum disconnecting set} of $\Gamma$.

The \emph{vertex connectivity} (or simply \emph{connectivity}) of a graph $\Gamma$, denoted by $\kappa({\Gamma})$, is the minimum number of vertices whose removal results in a disconnected or trivial graph. The \emph{edge-connectivity} of a graph $\Gamma$, denoted by $\kappa'({\Gamma})$, is the minimum number of edges whose removal results in a disconnected or trivial graph. So, the connectivity and edge-connectivity of disconnected graphs or the trivial graph are always $0$.

 If $A$ is a vertex set or an edge set of a graph $\Gamma$, then the subgraph obtained by deleting $A$ from $\Gamma$ will be denoted by $\Gamma-A$. If $A=\{ x \}$, $\Gamma-A$ is simply written as $\Gamma - x$. The \emph{neighbourhood} $N(v)$ of a vertex $v$ in a simple graph $\Gamma$ is the set of vertices which are adjacent to $v$. If $A, B \subseteq V(\Gamma)$, then the set of all edges having one end in $A$ and the other in $B$ is denoted by $E[A,B]$. If $A=\{v\}$, we simply write $E[A,B]=E[v,B]$.

For a positive integer $n$, the number of positive integers that do not exceed $n$ and are relatively prime to $n$ is denoted by $\phi(n)$. The function $\phi$ is known as \emph{Euler's phi function}. If an integer $n>1$ has the prime factorization $p_1^{\alpha_1} p_2^{\alpha_2}\ldots p_r^{\alpha_r}$, then  $\phi(n)=\displaystyle\prod_{i=1}^r (p_i^{\alpha_i}-p_i^{\alpha_i-1})=n \prod_{i=1}^r \left(1- \dfrac{1}{p_i}\right)$ (cf. \cite[Theorem 7.3]{burton2006elementary}).

In a group $G$, the cyclic subgroup generated by $x \in G$ is denoted by $\langle x \rangle$. For a prime $p$, a $p$-group is a finite group
whose order is some power of $p$. If two finite groups are isomorphic, their corresponding power graphs are isomorphic. Since a cyclic group of order $n$ is isomorphic to the additive group of integers modulo $n$, written as $\mathbb{Z}_n = \{\overline{0},\overline{1}, \ldots, \overline{n-1}\}$, we prove the results for $\mathbb{Z}_n$ instead of an arbitrary cyclic group.

For $n \in \mathbb{N}$, we denote the set consisting of $\overline{0}$ and generators of $\mathbb{Z}_n$ by $\mathcal{S}(\mathbb{Z}_n)$ i.e. $\mathcal{S}(\mathbb{Z}_n)=\left \{\overline{a}\in \mathbb{Z}_n :1 \leq a<n, \gcd(a, n)=1 \right \} \cup  \{\overline{0} \}$. We further write $\widetilde{\mathbb{Z}}_n=\mathbb{Z}_n-\mathcal{S}(\mathbb{Z}_n)$ and $\mathcal{\widetilde{G}}(\mathbb{Z}_n)=\mathcal{G}(\mathbb{Z}_n)-\mathcal{S}(\mathbb{Z}_n)$, so that $V(\widetilde{\mathcal{G}}(\mathbb{Z}_n))=\widetilde{\mathbb{Z}}_n$. We notice that $\widetilde{\mathbb{Z}}_n$ is empty if and only if $n \in \mathbb{N}$ is prime.

\begin{remark}\label{szn-adj}
For $n \in \mathbb{N}$,  each $\overline{a} \in \mathcal{S}(\mathbb{Z}_n)$ is adjacent to every other element of $\mathcal{G}(\mathbb{Z}_n)$ and hence $\deg(\overline{a})=n-1$ for all $\overline{a} \in \mathcal{S}(\mathbb{Z}_n)$.
\end{remark}

The following lemma gives us a formula to compute degrees of elements $\widetilde{\mathbb{Z}}_n$.

\begin{lemma}[{\cite[Lemma 3.4]{MR3200118}}]\label{verdeg}
Suppose $n \in \mathbb{N}$ is not a prime power and $\overline{a} \in \widetilde{\mathbb{Z}}_n$. If $b=\gcd(a,n)$, then in $\mathcal{G}(\mathbb{Z}_n)$,
$$ \deg(\overline{a})=\frac{n}{b} + \sum \limits_{d|b, d\neq b}\phi \left( \frac{n}{d} \right)-1.$$
\end{lemma}

Chakrabarty et al. have obtained the following basic properties of power graphs of finite groups.

\begin{lemma}[\cite{GhoshSensemigroups}]\label{CompleteCond}
Let $G$ be a finite group.
\begin{enumerate}[\rm(i)]
\item The power graph $\mathcal{G}(G)$ is always connected.
\item The power graph $\mathcal{G}(G)$ is complete if and only if $G$ is a
cyclic group of order $1$ or $p^\alpha$, for some prime number $p$ and for some $\alpha \in \mathbb{N}$.
\end{enumerate}
\end{lemma}

 When $n$ is composite, the following lemma provides a characterization for which $\mathcal{\widetilde{G}}(\mathbb{Z}_n)$ is connected.

\begin{lemma}[{\cite[Proposition 2.4]{ConPowerGr17}}]\label{SepSetLemma}
Let $n \in \mathbb{N}$ be a composite number. Then $\mathcal{\widetilde{G}}(\mathbb{Z}_n)$ is disconnected if and only if $n$ is a product of two distinct primes.
\end{lemma}

\section{Equality of edge-connectivity and minimum degree of power graphs}
\label{s-eq-md-ec}

In this section, we show that the edge-connectivity and minimum degree of power graphs of finite groups are equal, and consequently present a way to find the minimum disconnecting sets in terms of the neighbourhoods of the vertices having minimum degree.

It was proved in \cite{whitney1932congruent} that $\kappa(\Gamma) \leq \kappa'(\Gamma) \leq \delta(\Gamma)$ for any finite simple graph $\Gamma$. In particular, we have the following for power graphs.

\begin{lemma}\label{ConDeglneqPower}
If $G$ is a finite group, then $\kappa(\mathcal{G}(G))\leq \kappa'(\mathcal{G}(G)) \leq \delta(\mathcal{G}(G))$.
\end{lemma}

 Let $G$ is a finite group. If $|G| \leq 2$, then trivially $\kappa'(\mathcal{G}(G))=\delta(\mathcal{G}(G))$. If $|G|\geq 3$, then every pair $x$, $y$ of distinct vertices is connected by the path $x,e,y$ of length $2$ in $\mathcal{G}(G)$. Combining this with the fact that if $\Gamma$ is a graph with $diam(\Gamma)\leq 2$, then $\kappa'(\Gamma)=\delta(\Gamma)$  (see \cite{plesnik1975critical}), we obtain the following.

\begin{theorem}\label{EdgeConDegEq}
If $G$ is a finite group, then $\kappa'(\mathcal{G}(G))=\delta(\mathcal{G}(G))$.
\end{theorem}

Let $G$ be a finite group, $|G| \geq 2$ and $x \in G$ be such that $\delta(\mathcal{G}(G)) = deg(x)$. Observe that $E[x, N(x)]$ is a disconnecting set of $\mathcal{G}(G)$. Moreover, it follows from \Cref{EdgeConDegEq} that $|E[x, N(x)]|=\deg(x)=\kappa'(\mathcal{G}(G))$. Thus we have the following lemma.

\begin{lemma}\label{MinDisconSet}
Let $G$ be a finite group, $|G| \geq 2$ and $x \in G$ be such that $\delta(\mathcal{G}(G))=deg(x)$. Then $E[x, N(x)]$ is a minimum disconnecting set of  $\mathcal{G}(G)$.
\end{lemma}

\section{Minimum degree of $\mathcal{G}(\mathbb{Z}_n)$}
\label{s-cyclic}

In this section, we first give a sufficient condition for equality of degrees of elements of power graphs of finite groups. We then observe that  $\delta(\mathcal{G}(\mathbb{Z}_n))$ is the degree of one of the proper divisors of $n$ (cf. \Cref{MinDegVertex}). We show that $\phi(n) + 1$ is a sharp lower bound for $\delta(\mathcal{G}(\mathbb{Z}_n))$. Further, we obtain some inequalities involving degrees of various elements of $\mathbb{Z}_n$ (cf. \Cref{degcompare}). Subsequently, we determine $\delta(\mathcal{G}(\mathbb{Z}_n))$ when $n$ has two prime factors or $n$ is a product of at most four distinct primes (cf. \Cref{MinDegValue}). We conclude this section by giving two sharp upper bounds of $\delta(\mathcal{G}(\mathbb{Z}_n))$.

\begin{lemma}\label{OrderDeg}
Let $G$ be a  finite group and $x,y \in G$. If $\langle x \rangle =\langle y \rangle$, then $\deg(x)=\deg(y)$ in $\mathcal{G}(G)$.
\end{lemma}

\begin{proof}
Suppose $\langle x \rangle =\langle y \rangle$, so that $x$ and $y$ are powers of each other in $G$. In particular, $x$ and $y$ are adjacent in $\mathcal{G}(G)$. Further, any $z \in G$ distinct from $x, y$ is adjacent to $x$ if and only if it is adjacent to $y$ in $\mathcal{G}(G)$. Hence $\deg(x)=\deg(y)$.
\end{proof}

The converse need not be true though. For example, in $\mathcal{G} ( \mathbb{Z}_{12} )$, $\deg\left(\overline{2}\right) = \deg\left(\overline{6}\right)=9$, but $\langle\overline{2} \rangle \neq \langle \overline{6} \rangle$.\

 \begin{lemma}\label{MinDegVertex}
If $n \in \mathbb{N}$ is a composite number, then there exists $\overline{c} \in \mathbb{Z}_n$ satisfying $1<c<n$ and $c|n$ such that $\delta(\mathcal{G}(\mathbb{Z}_n))=\deg\left(\overline{c}\right)$.
\end{lemma}

\begin{proof}
Since $n$ is a composite number, $\widetilde{\mathbb{Z}}_n \neq \emptyset$. First of all, $\deg(\overline{a}) \leq n-1$ for all $\overline{a} \in \mathbb{Z}_n$, and by \Cref{szn-adj}, $\deg(\overline{b}) = n-1$ for all $\overline{b} \in \mathcal{S}(\mathbb{Z}_n)$.  So there exists $\overline{a} \in \widetilde{\mathbb{Z}}_n$ such that $\delta(\mathcal{G}(\mathbb{Z}_n))=\deg\left( \overline{a} \right)$. Now take $c =\gcd(a, n)$, so that $c | n$ and $1 < c < n$. Moreover, $\langle \overline{c} \rangle =\langle \overline{a} \rangle$ (cf. \cite[Theorem 4.2]{Gallian}), and hence by \Cref{OrderDeg}, $\delta(\mathcal{G}(\mathbb{Z}_n))=\deg\left(\overline{c}\right)$.
\end{proof}

The next lemma is a consequence of \Cref{CompleteCond}(ii).

\begin{lemma}\label{MinDegpPower}
For a finite group $G$, $\delta(\mathcal{G}(G))=|G|-1$ if and only if $G$ is a cyclic group of order $1$ or $p^\alpha$ for some prime number $p$ and $\alpha \in \mathbb{N}$.
\end{lemma}

\begin{theorem}\label{MinDegBasic}
For $n \in \mathbb{N}$, we have the following:
\begin{enumerate}[\rm(i)]
\item If $n$ is composite, then $\delta(\mathcal{G}(\mathbb{Z}_n))=\phi(n)+1+\delta(\mathcal{\widetilde{G}}(\mathbb{Z}_n))$. Consequently, $\delta(\mathcal{G}(\mathbb{Z}_n)) \geq \phi(n)+1$.
\item $\delta(\mathcal{G}(\mathbb{Z}_n)) = \phi(n)+1$ if and only if $n=2p$ for some prime $p\geq 3$.
\end{enumerate}
\end{theorem}

\begin{proof}

\noindent
(i)  Let $n \in \mathbb{N}$ be a composite number. In view of \Cref{MinDegVertex}, it is enough to consider only elements of $\widetilde{\mathbb{Z}}_n$. Any $\overline{a} \in \widetilde{\mathbb{Z}}_n$ is adjacent to all elements of $\mathcal{S}(\mathbb{Z}_n)$. Since $|\mathcal{S}(\mathbb{Z}_n)| = \phi(n)+1$, we have $\deg_{\mathcal{G}(\mathbb{Z}_n)}(\overline{a})=\deg_{\mathcal{\widetilde{G}}(\mathbb{Z}_n)}(\overline{a})+ \phi(n)+1$. Thus the proof follows.

\noindent
(ii) Let $n=2p$ for some prime $p \geq 3$. Then by \Cref{SepSetLemma}, $\mathcal{\widetilde{G}}(\mathbb{Z}_n)$ is disconnected, and its component induced by $\langle \overline{p} \rangle^*$ has $\overline{p}$ as its only vertex. Then $\delta(\mathcal{\widetilde{G}}(\mathbb{Z}_n))=0$, and hence by (i), $\delta(\mathcal{G}(\mathbb{Z}_n)) = \phi(n)+1$.

Conversely, let $\delta(\mathcal{G}(\mathbb{Z}_n)) = \phi(n)+1$. If $n$ is prime, then $\delta(\mathcal{G}(\mathbb{Z}_n))=n-1 \neq \phi(n)+1$. So $n$ is a composite number. Then by (i), $\delta(\mathcal{\widetilde{G}}(\mathbb{Z}_n))=0$, and hence $\mathcal{\widetilde{G}}(\mathbb{Z}_n)$ is disconnected. Accordingly, by \Cref{SepSetLemma}, $n$ is a product of two distinct primes; say $n=pq$. It is easy to see that subgraphs induced by $\langle p \rangle^*$ and $\langle q \rangle^*$ are the only components of $\mathcal{\widetilde{G}}(\mathbb{Z}_n)$. If $p,q \geq 3$, then $|\langle p \rangle^*|,|\langle q \rangle^*| \geq 2$, so that $\delta(\mathcal{\widetilde{G}}(\mathbb{Z}_n))\geq 1$; a contradiction. So exactly one of $p$ or $q$ is $2$, say $q=2$. Hence, $n=2p$.
\end{proof}

 We now obtain certain relations between the degrees of vertices of $\mathcal{G}(\mathbb{Z}_n)$.

\begin{proposition}\label{degcompare}
 Let $n=p_1^{\alpha_1}p_2^{\alpha_2}\ldots p_r^{\alpha_r}$, $r \geq 2$, $p_1 < p_2 < \cdots < p_r$ are primes and $\alpha_i \in \mathbb{N}$ for $1 \leq i \leq r$. Then the following inequalities hold in $\mathcal{G}(\mathbb{Z}_n)$:
 \begin{enumerate}[\rm(i)]
 \item $\deg\left(\overline{p_1^{\alpha_1}}\right) \geq \deg\left(\overline{p_r^{\alpha_r}}\right)$.
 \item $\deg\left(\overline{p_i^{\gamma}}\right) \geq \deg\left(\overline{p_i^{\beta}}\right)$ for all $1 \leq i \leq r$ and $1 \leq \gamma < \beta \leq \alpha_i$.
  \item $\deg\left(\overline{p_i^\beta}\right) \geq \deg\left(\overline{p_j^\beta}\right)$ for all $1 \leq i < j \leq r$ and $1 \leq \beta \leq \min \{\alpha_i, \alpha_j\}$.
  \item $\deg\left(\overline{p_1^{\beta_1}p_2^{\beta_2}\ldots p_r^{\beta_r}}\right) \geq \deg\left(\overline{p_2^{\beta_2}\ldots p_r^{\beta_r}}\right)$ for $\displaystyle\sum_{i=1}^r \beta_i < \sum_{i=1}^r \alpha_i$, where $1 \leq \beta_i \leq \alpha_i$ for all $1 \leq i \leq r$.
 \end{enumerate}
\end{proposition}

\begin{proof}

\noindent
(i) Let $m=\dfrac{n}{p_1^{\alpha_1}p_2^{\alpha_2}}$.
    \begin{align*}
    & \deg\left(\overline{p_1^{\alpha_1}}\right) - \deg\left(\overline{p_r^{\alpha_r}}\right)\\
     & = \frac{n}{p_1^{\alpha_1}}+\sum \limits_{k=0}^{\alpha_1-1}\phi \left( \frac{n}{p_1^{k}}\right)  - \left \{\frac{n}{p_r^{\alpha_r}} +\sum \limits_{k=0}^{\alpha_r-1}\phi \left( \frac{n}{p_r^{k}}\right) \right \}\\
     & = \phi (m) \left [\phi(p_r^{\alpha_r}) \sum \limits_{k=0}^{\alpha_1-1}\phi \left( p_1^{\alpha_1-k}\right) - \phi(p_1^{\alpha_1}) \sum \limits_{k=0}^{\alpha_r-1}\phi \left( p_r^{\alpha_r-k}\right) \right ] \\
     & \qquad\qquad   + m \left(p_r^{\alpha_r}-p_1^{\alpha_1}\right) \\
     \end{align*}

     \begin{align*}
     & = \phi(m) \left [(p_r^{\alpha_r}-p_r^{\alpha_r-1})(p_1^{\alpha_1}-1)  - (p_1^{\alpha_1}-p_1^{\alpha_1-1}) (p_r^{\alpha_r}-1) \right ]\\
     & \qquad\qquad   + m\{p_1^{\alpha_1}(p_r^{\alpha_r}-1)- p_r^{\alpha_r}(p_1^{\alpha_1}-1)\}\\
      & = (p_r^{\alpha_r}-1)\left [  p_1^{\alpha_1-1}\phi(m) +p_1^{\alpha_1} \{m-\phi(m)\} \right ]\\
        & \qquad\qquad  - (p_1^{\alpha_1}-1) \left [ p_r^{\alpha_r-1}\phi(m) + p_r^{\alpha_r}\{m-\phi(m)\} \right ]\\
           & \geq (p_r^{\alpha_r}-1)\left [  p_1^{\alpha_1-1}\phi(m) +p_1^{\alpha_1} \{m-\phi(m)\} \right ]\\
        & \qquad\qquad  - p_1^{\alpha_1} \left [ p_r^{\alpha_r-1}\phi(m) + p_r^{\alpha_r}\{m-\phi(m)\} \right ]\\
      & = (p_r^{\alpha_r}p_1^{\alpha_1-1}-p_1^{\alpha_1} p_r^{\alpha_r-1})\phi(m) - \left [  p_1^{\alpha_1-1}\phi(m) +p_1^{\alpha_1} \{m-\phi(m)\} \right ] \\
      & = p_1^{\alpha_1-1}\left [p_r^{\alpha_r-1}(p_r-p_1)\phi(m) - \left [\phi(m) +p_1 \{m-\phi(m)\} \right ] \right ]\\
      & \geq p_1^{\alpha_1-1}\left [(p_r-p_1)\phi(m) - \left [\phi(m) +p_1 \{m-\phi(m)\} \right ] \right ]\\
      & =p_1^{\alpha_1-1} \left \{(p_r-1)\phi(m) - p_1m\right \} \numberthis \label{DegEqZn}
    \end{align*}

Now, if $r=2$, then $m=\dfrac{n}{p_1^{\alpha_1}p_2^{\alpha_2}}=1$. Hence, from \eqref{DegEqZn}, \[\deg\left(\overline{p_1^{\alpha_1}}\right) - \deg\left(\overline{p_r^{\alpha_r}}\right) \geq p_1^{\alpha_1-1} \left(p_r-1 - p_1\right) \geq 0.\]

If $r>2$, then $m=\dfrac{n}{p_1^{\alpha_1}p_r^{\alpha_r}}=\displaystyle\prod_{i=2}^{r-1}p_i^{\alpha_i}$.
Hence, from \eqref{DegEqZn},
\begin{align*}
 \deg\left(\overline{p_1^{\alpha_1}}\right) - \deg\left(\overline{p_r^{\alpha_r}}\right) \geq & \prod_{i=1}^{r-1}p_i^{\alpha_i-1} \left\{\prod_{i=2}^{r}(p_i-1) - \prod_{i=1}^{r-1}p_i \right \} \geq 0,\\
& \text{ since } p_{i+1}-1 \geq p_i \text{ for all } 1 \leq i \leq r-1.
\end{align*}

\noindent
(ii)
$\deg\left(\overline{p_i^{\gamma}}\right) - \deg\left(\overline{p_i^{\beta}}\right)$
\begin{align*}
& = \frac{n}{p_i^{\gamma}} + \sum \limits_{k=0}^{\gamma-1}\phi \left( \frac{n}{p_i^{k}}\right) - \frac{n}{p_i^{\beta}} - \sum \limits_{k= 0}^{\beta-1}\phi \left( \frac{n}{p_i^{k}}\right) \\
  & = \frac{n}{p_i^{\gamma}}- \frac{n}{p_i^{\beta}} - \sum \limits_{k= \gamma}^{\beta-1}\phi \left( \frac{n}{p_i^{k}}\right) \\
  & = \frac{n}{p_i^{\alpha_i}} \left( p_i^{\alpha_i-\gamma}- p_i^{\alpha_i-\beta} \right) - \phi\left(\frac{n}{p_i^{\alpha_i}}\right) \sum \limits_{k=\gamma}^{\beta-1}\phi \left( p_i^{\alpha_i-k}\right) \\
  & = \frac{n}{p_i^{\alpha_i}} \left( p_i^{\alpha_i-\gamma}- p_i^{\alpha_i-\beta} \right) - \phi\left(\frac{n}{p_i^{\alpha_i}}\right) \left( p_i^{\alpha_i-\gamma}- p_i^{\alpha_i-\beta} \right)\\
  & = \left( p_i^{\alpha_i-\gamma}- p_i^{\alpha_i-\beta} \right) \left \{\frac{n}{p_i^{\alpha_i}} - \phi\left(\frac{n}{p_i^{\alpha_i}}\right)\right \} \geq 0 \hspace{5pt} (\text{since } \gamma < \beta)
\end{align*}

\noindent
(iii) $\deg\left(\overline{p_i^\beta}\right) - \deg\left(\overline{p_j^\beta}\right) = \dfrac{n}{p_i^\beta}- \dfrac{n}{p_j^\beta} + \sum \limits_{k=0}^{\beta-1} \left \{\phi \left( \dfrac{n}{p_i^{k}}\right) - \phi \left( \dfrac{n}{p_j^k}\right) \right \}.$

Since $p_i<p_j$, we have $\dfrac{n}{p_i^{\beta}} > \dfrac{n}{p_j^{\beta}}$. Further $\alpha_i, \alpha_j \geq \beta$, so for all $0 \leq k \leq \beta-1$, we have $\phi \left( \dfrac{n}{p_i^{k}}\right)- \phi \left( \dfrac{n}{p_j^{k}}\right)=  \left( \dfrac{n}{p_i^{k}}- \dfrac{n}{p_j^{k}}\right) \prod \limits_{l=1}^{r} \left(1- \dfrac{1}{p_l}\right) \geq 0$. Thus the proof follows.

\noindent
(iv) $\deg\left(\overline{p_1^{\beta_1}p_2^{\beta_2}\ldots p_r^{\beta_r}}\right) - \deg\left(\overline{p_2^{\beta_2}\ldots p_r^{\beta_r}}\right)$
\begin{align*}
& = \frac{n}{p_1^{\beta_1}p_2^{\beta_2}\ldots p_r^{\beta_r}} + \sum \limits_{\substack{ 0 \leq \sum \limits_{i=1}^r \gamma_i < \sum \limits_{i=1}^r \beta_i ,\\ 0 \leq \gamma_i \leq \beta_i \forall 1 \leq i \leq r}}\phi \left( \frac{n}{p_1^{\gamma_1}p_2^{\gamma_2}\ldots p_r^{\gamma_r}}\right)\\
&\quad\quad\quad\quad -  \left \{ \frac{n}{p_2^{\beta_2}\ldots p_r^{\beta_r}} + \sum \limits_{\substack{ 0 \leq \sum \limits_{i=2}^r \gamma_i < \sum \limits_{i=2}^r \beta_i,\\ 0 \leq \gamma_i \leq \beta_i \forall 2 \leq i \leq r}}\phi \left( \frac{n}{p_2^{\gamma_2}\ldots p_r^{\gamma_r}}\right) \right \} \\
 & = \frac{n}{p_1^{\beta_1}p_2^{\beta_2}\ldots p_r^{\beta_r}} + \sum \limits_{0 \leq \gamma_i \leq \beta_i \forall 1 \leq i \leq r }\phi \left( \frac{n}{p_1^{\gamma_1}p_2^{\gamma_2}\ldots p_r^{\gamma_r}}\right) -\phi \left( \frac{n}{p_1^{\beta_1}p_2^{\beta_2}\ldots p_r^{\beta_r}}\right) \\
 &\quad\quad\quad\quad - \left \{ \frac{n}{p_2^{\beta_2}\ldots p_r^{\beta_r}} + \sum \limits_{0 \leq \gamma_i \leq \beta_i \forall 2 \leq i \leq r }\phi \left( \frac{n}{p_2^{\gamma_2}\ldots p_r^{\gamma_r}}\right) - \phi \left( \frac{n}{p_2^{\beta_2}\ldots p_r^{\beta_r}}\right) \right \}\\
  \end{align*}
 \begin{align*}
 & = \sum \limits_{\substack{1 \leq \gamma_1 \leq \beta_1,\\ 0 \leq \gamma_i \leq \beta_i \forall 2 \leq i \leq r}}\phi \left( \frac{n}{p_1^{\gamma_1}p_2^{\gamma_2}\ldots p_r^{\gamma_r}}\right)+  \phi(m) \left \{\phi (p_1^{\alpha_1})-\phi (p_1^{\alpha_1-\beta_1})\right \} \\
 & \quad\quad\quad\quad + m \left(p_1^{\alpha_1-\beta_1}-p_1^{\alpha_1}\right) \left ( \text{taking } m =p_2^{\alpha_2-\beta_2}\ldots p_r^{\alpha_r-\beta_r} \right)\\
 & \geq  \left \{ \sum \limits_{0 \leq \gamma_i \leq \beta_i \forall 2 \leq i \leq r} \phi\left(p_2^{\alpha_2-\gamma_2}\ldots p_r^{\alpha_r-\gamma_r}\right) \right \}\sum \limits_{1 \leq \gamma_1 \leq \beta_1} \phi \left(p_1^{\alpha_1-\gamma_1}\right) +m \left(p_1^{\alpha_1-\beta_1}-p_1^{\alpha_1}\right)\\
& \geq   \phi\left(p_2^{\alpha_2}\ldots p_r^{\alpha_r}\right)\sum \limits_{1 \leq \gamma_1 \leq \beta_1} \phi \left(p_1^{\alpha_1-\gamma_1}\right) +m \left(p_1^{\alpha_1-\beta_1}-p_1^{\alpha_1}\right)\\
&= \prod_{i=2}^r \left \{ p_i^{\alpha_i-1}(p_i-1) \right \} \sum \limits_{1 \leq \gamma_1 \leq \beta_1} \phi \left(p_1^{\alpha_1-\gamma_1}\right) - m \left(p_1^{\alpha_1}-p_1^{\alpha_1-\beta_1}\right)\\
&\geq m \left \{ \prod_{i=2}^r(p_i-1)\sum \limits_{1 \leq \gamma_1 \leq \beta_1} \phi \left(p_1^{\alpha_1-\gamma_1}\right) - \left(p_1^{\alpha_1}-p_1^{\alpha_1-\beta_1}\right)\right \}
\end{align*}

Thus it is enough to show that \[\tau := \prod_{i=2}^r(p_i-1)\sum \limits_{1 \leq \gamma_1 \leq \beta_1} \phi \left(p_1^{\alpha_1-\gamma_1}\right) - \left(p_1^{\alpha_1}-p_1^{\alpha_1-\beta_1}\right) \geq 0.\]

We have $\beta_1 \leq \alpha_1$. First take $\alpha_1=\beta_1$. Then $\sum \limits_{1 \leq \gamma_1 \leq \beta_1} \phi \left(p_1^{\alpha_1-\gamma_1}\right)=p_1^{\alpha_1-1}$, and we have

\begin{align*}
\tau & =\prod_{i=2}^r(p_i-1) p_1^{\alpha_1-1} - \left(p_1^{\alpha_1}-1\right)\\
& \geq (p_2-1)p_1^{\alpha_1-1} -p_1^{\alpha_1}+1\\
      & \geq p_1^{\alpha_1} -p_1^{\alpha_1}+1 > 0.\\
      \end{align*}

Now we take $\alpha_1 > \beta_1$. In this case, $\sum \limits_{1 \leq \gamma_1 \leq \beta_1} \phi \left(p_1^{\alpha_1-\gamma_1}\right)=p_1^{\alpha_1-1}-p_1^{\alpha_1-\beta_1-1}$. Hence,

\begin{align*}
 \tau &=(p_2-1)\ldots (p_r-1)\left(p_1^{\alpha_1-1}-p_1^{\alpha_1-\beta_1-1}\right) - \left(p_1^{\alpha_1}-p_1^{\alpha_1-\beta_1}\right)\\
 &=\left(p_1^{\alpha_1-1}-p_1^{\alpha_1-\beta_1-1}\right) \left \{(p_2-1)\ldots (p_r-1) - p_1 \right \} \geq 0.\\
\end{align*}

\end{proof}

\begin{theorem}\label{MinDegValue}
Let $n \in \mathbb{N}$ and $p_1 < p_2 < p_3 < p_4$ be prime numbers.
\begin{enumerate}[\rm(i)]
\item If $n=p_1^{\alpha_1} p_2^{\alpha_2}$, $\alpha_1, \alpha_2 \in \mathbb{N}$, then $\overline{p_2^{\alpha_2}}$ has the minimum degree among all vertices in $\mathcal{G}(\mathbb{Z}_n)$, and $\delta(\mathcal{G}(\mathbb{Z}_n)) = (p_2^{\alpha_2}-1) \phi(p_1^{\alpha_1})+ p_1^{\alpha_1} -1$.
\item If $n=p_1 p_2 p_3$, then $\overline{p_3}$ has the minimum degree among all vertices in $\mathcal{G}(\mathbb{Z}_n)$, and $\delta(\mathcal{G}(\mathbb{Z}_n)) = \phi(n) + p_1p_2 -1.$

\item Let $n = p_1 p_2 p_3p_4$. If $n$ is odd or $p_4 \geq p_3 + \displaystyle \frac{2(p_3-1)}{p_2-1}$, then $\overline{p_4}$ has the minimum degree among all vertices in $\mathcal{G}(\mathbb{Z}_n)$, and $\delta(\mathcal{G}(\mathbb{Z}_n)) = \phi(n) + p_1p_2p_3-1$. Otherwise, $\overline{p_3p_4}$ has the minimum degree among all vertices in $\mathcal{G}(\mathbb{Z}_n)$, and $\delta(\mathcal{G}(\mathbb{Z}_n)) = (p_2-1)(p_3p_4+1)+1$.
\end{enumerate}
\end{theorem}

\begin{proof}
In view of \Cref{MinDegVertex}, in order to determine $\delta(\mathcal{G}(\mathbb{Z}_n))$, it is sufficient to compare degree of vertices $\overline{c}$, where $c > 1$ is a proper divisor of $n$.

\noindent
(i) Consider $\beta_1,\beta_2 \in \mathbb{N}$ with $1 \leq \beta_i \leq \alpha_i$ for $i=1,2$.

By \Cref{degcompare}(ii),(i), we have $\deg\left(\overline{p_1^{\beta_1}}\right) \geq \deg\left(\overline{p_1^{\alpha_1}}\right) \geq \deg\left(\overline{p_2^{\alpha_2}}\right)$.

By \Cref{degcompare}(ii), we have $\deg\left(\overline{p_2^{\beta_2}}\right) \geq \deg\left(\overline{p_2^{\alpha_2}}\right)$.

By \Cref{degcompare}(iv),(ii), we have $\deg\left(\overline{p_1^{\beta_1}p_2^{\beta_2}}\right) \geq \deg\left(\overline{p_2^{\beta_2}}\right) \geq \deg\left(\overline{p_2^{\alpha_2}}\right)$.

Thus $\overline{p_2^{\alpha_2}}$ has the minimum degree among all vertices in $\mathcal{G}(\mathbb{Z}_n)$, and by \Cref{verdeg}, $\delta(\mathcal{G}(\mathbb{Z}_n)) = (p_2^{\alpha_2}-1) \phi(p_1^{\alpha_1})+ p_1^{\alpha_1} -1$.\\

\noindent
(ii) Let $i,j,k$ be a permutation of $1,2,3$ with $i<j$.
\begin{align*}
\deg\left(\overline{p_ip_j}\right)-\deg\left(\overline{p_j}\right) &= p_k+\phi(p_ip_k)+\phi(p_jp_k)-p_ip_k\\
&= (p_i-1)(p_k-1)+(p_j-1)(p_k-1)-p_k(p_i-1)\\
&= (p_j-1)(p_k-1)-(p_i-1) \geq 0,  \text{ since } p_i < p_j.\\
\end{align*}

Further, by \Cref{degcompare}(iii), $\deg\left(\overline{p_1}\right)\geq \deg\left(\overline{p_2}\right) \geq \deg\left(\overline{p_3}\right)$. Hence $\overline{p_3}$ has the minimum degree among all vertices in $\mathcal{G}(\mathbb{Z}_n)$. Consequently, by \Cref{verdeg}, $\delta(\mathcal{G}(\mathbb{Z}_n)) = \deg\left(\overline{p_3}\right) = \phi(n)+p_1p_2-1.$\\

\noindent
(iii) Let $i,j,k,l$ be a permutation of $1,2,3,4$.\\

For $i<j<k$, we have\\

$\deg\left(\overline{p_ip_jp_k}\right) - \deg\left(\overline{p_jp_k}\right)$
\begin{align*}
& = p_l+\sum \limits_{\substack{d|p_ip_jp_k\\ d \neq p_ip_jp_k}}\phi \left( \frac{n}{d} \right)- \left \{p_ip_l+\sum \limits_{\substack{d|p_jp_k\\ d \neq p_jp_k}}\phi \left( \frac{n}{d} \right) \right \}\\
&= p_l+ \phi\left(\frac{n}{p_ip_j}\right)+\phi\left(\frac{n}{p_ip_k}\right)+ \phi\left(\frac{n}{p_jp_k}\right)+\phi\left(\frac{n}{p_i}\right) - p_ip_l\\
&= (p_l-1)\left \{(p_k-1)+(p_j-1)+(p_i-1)+(p_j-1)(p_k-1)\right \}-p_l(p_i-1)\\
&= (p_l-1)\left \{(p_k-1)+(p_j-1)+(p_j-1)(p_k-1)\right \}-(p_i-1)\\
&\geq (p_l-1)(p_j-1)-(p_i-1) \geq 0, \hspace{4pt} \text{ since } p_i<p_j.
\end{align*}

Now take $i<j$ with no condition on $k$ and $l$.
\begin{align*}
\deg\left(\overline{p_ip_j}\right) & -\deg\left(\overline{p_j}\right) = p_kp_l+\phi(p_ip_kp_l)+\phi(p_jp_kp_l)-p_ip_kp_l\\
&= (p_i-1)(p_k-1)(p_l-1)+(p_j-1)(p_k-1)(p_l-1)-p_lp_k(p_i-1)\\
&= (p_j-1)(p_k-1)(p_l-1)-(p_i-1)(p_k +p_l-1) \numberthis \label{MinDegEq3}\\
\end{align*}
Since $k$ and $l$ can be interchanged in \eqref{MinDegEq3}, without loss of generality, let $p_k < p_l$.

If $n$ is odd, then $p_k > 2$, and hence

\begin{align*}
&\deg\left(\overline{p_ip_j}\right)-\deg(\overline{p_j})\\
& \geq (p_i-1) \{(p_k-1)(p_l-1)-(p_k +p_l-1)\} \hspace{4pt} (\text{since } p_i<p_j)\\
& = (p_i-1) \{(p_k-2)(p_l-1)-p_k\} \\
& \geq (p_i-1) \{(p_l-1)-p_k\} \geq 0, \text{since } p_k > 2 \text{ and } p_k < p_l.
\end{align*}

Now let $n$ be even i.e. $p_1 = 2$. If $p_k > 2$, then from \eqref{MinDegEq3}, $\deg\left(\overline{p_ip_j}\right) \geq \deg\left(\overline{p_j}\right)$ as shown above.  Otherwise, $p_k = 2$, so that \eqref{MinDegEq3} becomes

\begin{equation}\label{MinDegEq4}
\deg\left(\overline{p_ip_j}\right)  -\deg\left(\overline{p_j}\right) = (p_j-1)(p_l-1)-(p_i-1)(p_l+1)
\end{equation}

In \eqref{MinDegEq4}, let $p_i \neq p_3$. Since $i < j$, $p_i$ cannot be $p_4$. Moreover, $p_k=p_1 =2$, so we have $p_i = p_2$. As a result, $p_l > p_2$. Then from \eqref{MinDegEq4},

\begin{align*}
\deg\left(\overline{p_ip_j}\right)  -\deg\left(\overline{p_j}\right) & =(p_l-1) \{ (p_j-1)-(p_2-1)\} - 2(p_2-1)\\
& \geq 2(p_l-1) - 2(p_2-1) \text{ (since } p_j -p_2 \geq 2)\\
& = 2(p_l- p_2) > 0  \text{ (since } p_l >p_2)
\end{align*}

Now take $p_i = p_3$ in \eqref{MinDegEq4}. Then $p_j = p_4$. We already have $p_k = 2$, and since $p_k < p_l$, we have $p_l=p_2$. Then from \eqref{MinDegEq4},\\
$\deg\left(\overline{p_ip_j}\right)  -\deg\left(\overline{p_j}\right)  = \deg\left(\overline{p_3p_4}\right)  -\deg\left(\overline{p_4}\right) = (p_4-1)(p_2-1)-(p_3-1)(p_2+1)$, and hence

\begin{equation}\label{MinDegEq5}
\deg\left(\overline{p_3p_4}\right) \geq \deg\left(\overline{p_4}\right) \Leftrightarrow p_4 \geq p_3 + \displaystyle \frac{2(p_3-1)}{p_2-1}.
\end{equation}

\noindent
\emph{Case 1:}  $n$ is odd or $p_4 \geq p_3 + \displaystyle \frac{2(p_3-1)}{p_2-1}$

As shown above, for all $1 \leq i<j<k \leq 4$,
\begin{equation}\label{Eq4prime}
\left (\overline{p_ip_jp_k}\right) \geq \deg\left(\overline{p_jp_k}\right) \geq \deg \left( \overline{p_k} \right).
\end{equation}
Further, it follows from \Cref{degcompare}(iii) that
\begin{equation}\label{Eq4prime2}
\deg\left(\overline{p_1}\right)\geq \deg\left(\overline{p_2}\right) \geq \deg\left(\overline{p_3}\right)\geq \deg\left(\overline{p_4}\right).
\end{equation}
 So we conclude that $\overline{p_4}$ has the minimum degree among all vertices in $\mathcal{G}(\mathbb{Z}_n)$. Consequently, by \Cref{verdeg}, $\delta(\mathcal{G}(\mathbb{Z}_n)) = \deg\left(\overline{p_4}\right) = \phi(n) + p_1p_2p_3-1$.

 \vspace{7pt}

\noindent
\emph{Case 2:} $n$ is even and $p_4 < p_3 + \displaystyle \frac{2(p_3-1)}{p_2-1}$

Then from \eqref{MinDegEq5}, $\deg\left(\overline{p_3p_4}\right) < \deg\left(\overline{p_4}\right)$, whereas all other inequalities in \eqref{Eq4prime}  and \eqref{Eq4prime2} hold. Thus $\overline{p_3p_4}$ has the minimum degree among all vertices in $\mathcal{G}(\mathbb{Z}_n)$. Thus by \Cref{verdeg}, $\delta(\mathcal{G}(\mathbb{Z}_n)) = (p_2-1)(p_3p_4+1)+1$.
\end{proof}

In view of \Cref{OrderDeg} and \Cref{MinDisconSet}, we have the following corollary of \Cref{MinDegValue}.

\begin{corollary}
Let $n \in \mathbb{N}$ and $p_1 < p_2 < p_3 < p_4$ be prime numbers.
\begin{enumerate}[\rm(i)]
\item If $n=p_1^{\alpha_1} p_2^{\alpha_2}$, $\alpha_1, \alpha_2 \in \mathbb{N}$, then for any $\overline{a} \in \left[ \overline{p_2^{\alpha_2}} \right]$,\\ $E \left[\overline{a}, \left \langle \overline{p_2^{\alpha_2}} \right \rangle \cup \bigcup \limits_{i=0}^{\alpha_2-1} \left[ \overline{p_2^i}\right]-\overline{a} \right]$ is a minimum disconnecting set of $\mathcal{G}(\mathbb{Z}_n)$.
\item If $n=p_1 p_2 p_3$, then for any $\overline{a} \in \left[ \overline{p_3} \right]$, $E \left[\overline{a}, \left \langle \overline{p_3} \right \rangle \cup \left [ \overline{1} \right ]-\overline{a} \right]$ is a minimum disconnecting set of $\mathcal{G}(\mathbb{Z}_n)$.

\item Let $n = p_1 p_2 p_3p_4$. If $n$ is odd or $p_4 \geq p_3 + \displaystyle \frac{2(p_3-1)}{p_2-1}$, then for any $\overline{a} \in \left[ \overline{p_4} \right]$, $E \left[\overline{a}, \left \langle \overline{p_4} \right \rangle \cup \left [ \overline{1} \right ]-\overline{a} \right]$ is a minimum disconnecting set of $\mathcal{G}(\mathbb{Z}_n)$. Otherwise, for any $\overline{b} \in \left[ \overline{p_3p_4} \right]$, $E \left[\overline{b}, \left \langle \overline{p_3p_4} \right \rangle \cup \left [ \overline{p_3} \right ] \cup \left [ \overline{p_4} \right ] \cup \left [ \overline{1} \right ] -\overline{b} \right]$ is a minimum disconnecting set of $\mathcal{G}(\mathbb{Z}_n)$.
\end{enumerate}
\end{corollary}

\begin{proposition}\label{mindeg}
 Let $n=p_1^{\alpha_1}p_2^{\alpha_2}\ldots p_r^{\alpha_r}$, $r \geq 2$, $p_1 < p_2 < \cdots < p_r$ be prime numbers and $\alpha_i \in \mathbb{N}$ for $1 \leq i \leq r$. Let
\begin{equation}\label{mindegEq}
\eta_1(n)= \frac{n}{p_r^{\alpha_r}} + (p_r^{\alpha_r}-1) \phi \left( \frac{n}{p_r^{\alpha_r}}\right)-1,
\end{equation}
and
\begin{equation}\label{mindegEq2}
\eta_2(n)= \frac{n}{p_{r-1}p_r} +\phi(n)+ \phi \left( \frac{n}{p_r}\right) + \phi \left( \frac{n}{p_{r-1}}\right)-1.
\end{equation}
Then $\eta_1(n)$ and $\eta_2(n)$ are sharp upper bounds of $\delta(\mathcal{G}(\mathbb{Z}_n))$.
\end{proposition}

\begin{proof}
By \Cref{verdeg},
\begin{equation}\label{mindegUB}
\deg\left(\overline{p_r^{{\alpha_r}}}\right) = \displaystyle \frac{n}{p_r^{{\alpha_r}}} + \sum \limits_{i=0}^{\alpha_r-1}\phi \left( \displaystyle \frac{n}{p_r^i} \right)-1=\displaystyle \frac{n}{p_r^{\alpha_r}} + (p_r^{\alpha_r}-1) \phi \left( \displaystyle \frac{n}{p_r^{\alpha_r}}\right)-1,
\end{equation}

and

\begin{equation}\label{mindegUB2}
\deg(\overline{p_{r-1}p_r})=\frac{n}{p_{r-1}p_r} +\phi(n)+ \phi \left( \frac{n}{p_r}\right) + \phi \left( \frac{n}{p_{r-1}}\right)-1.
\end{equation}

Thus $\eta_1(n)$ and $\eta_2(n)$ are upper bounds of $\delta(\mathcal{G}(\mathbb{Z}_n))$. Moreover, it follows from \Cref{MinDegValue}(i),(ii),(iii) that the bound $\eta_1(n)$ is sharp, and it follows from \Cref{MinDegValue}(iii) that the bound $\eta_2(n)$ is sharp.
  \end{proof}

Here is a consequence of \Cref{ConDeglneqPower} and \Cref{mindeg}.
\begin{corollary}
For $n \in \mathbb{N}$, $\eta_1(n)$ and $\eta_2(n)$ are upper bounds of  $\kappa(\mathcal{G}(\mathbb{Z}_n))$.
\end{corollary}

\section{Minimum degree of power graphs of abelian $p$-group, $D_n$ and $Q_n$}
\label{s-grps}

In this section, we find the minimum degree and minimum disconnecting sets of abelian $p$-groups, dihedral groups and dicyclic groups in respective subsections.

\subsection{Abelian $p$-groups}
\label{s-pgr}

By \cite[Theorem 11.1]{Gallian}, a finite abelian group $G$ is isomorphic to an unique direct product of cyclic groups of prime power order.  In this product, let $\sigma(G)$ be the number of cyclic groups and $\tau(G)$ be order of the smallest cyclic group.

\begin{theorem}\label{DegAbelianp}
Let $G$ be an abelian $p$-group for some prime $p$, then $\delta(\mathcal{G}(G))=\tau(G)-1$.
\end{theorem}

\begin{proof}
We have $G \cong H:=\mathbb{Z}_{p^{\alpha_1}} \times \mathbb{Z}_{p^{\alpha_2}} \times \ldots \times \mathbb{Z}_{p^{\alpha_r}}$ for some $r \in \mathbb{N}$, prime $p$ and $\alpha_i \in \mathbb{N}$ for all $1 \leq i \leq r$. Take $\alpha_t = \displaystyle\min_{1\leq i \leq r} \alpha_i$, so that $\rho(G)=p^{\alpha_t} - 1$.

 Since isomorphic groups have isomorphic power graphs, it is enough to show that $\delta(\mathcal{G}(H))=p^{\alpha_t}-1$. If $r=1$, the proof follows trivially. So for the rest of the proof, set $r \geq 2$.

Let $x =(\overline{a_1},\overline{a_2},\ldots,\overline{a_r}) \in H$. For any $1 \leq i \leq r$, if $\overline{a_i} \neq \overline{0}$, then $a_i=c_ip^{\beta_i}$ for some $c_i \in \mathbb{N}$,  $(c_i,p)=1$ and $\beta_i \in \mathbb{N} \cup \{0\}$. Take $\beta_s=\min\{\beta_i \mid 1 \leq i \leq r, \overline{a_i} \neq \overline{0}\}$, and define $y=\left(\overline{b_1},\overline{b_2},\ldots,\overline{b_r}\right)$ by
\begin{equation}
b_i = \begin{cases} c_ip^{\beta_i-\beta_s} &\mbox{if } \overline{a_i} \neq \overline{0} \\
0 & \mbox{if } \overline{a_i} = \overline{0} \end{cases}
\end{equation}

Then $\overline{b_s}=\overline{c_s}$, and hence $o(\overline{b_s})=p^{\alpha_s}$. Since $o(y)= \text{lcm}(o(\overline{b_1}),o(\overline{b_2}),\ldots,o(\overline{b_r}))$ \cite[Theorem 8.1]{Gallian}, we get  $o(y)\geq p^{\alpha_s}$. Additionally, $o(y)$ is a prime power, so that $\langle y \rangle$ is a clique \cite{GhoshSensemigroups}. Thus, since $x \in \langle y \rangle$, we have $\deg(x)\geq p^{\alpha_s}-1\geq p^{\alpha_t}-1$. Therefore,
\begin{equation}\label{DegAbelianpEq}
\delta(\mathcal{G}(H)) \geq p^{\alpha_t}-1.
\end{equation}

Now consider $z \in H$ with all components $\overline{0}$ except $t^{th}$, which is $\overline{1}$. Then $\langle z \rangle=\prod_{i=1}^{r}K_i$, where $K_t=\mathbb{Z}_{p^{\alpha_t}}$ and $K_i=\langle \overline{0} \rangle$ for all $1 \leq i \leq r$, $i \neq t$.
Thus $\deg(z) \geq p^{\alpha_t}-1$. We next show that $z$ is not adjacent to any element of $\in H - \langle z \rangle$.  If possible let $z$ is adjacent to some $w \in H - \langle z \rangle$. As $w \notin \langle z \rangle$, there exists $a \in \mathbb{N}$ such that $z=aw$. If $(a,p)=1$, then $\langle w \rangle=\langle z \rangle$; which is not possible, so that $p|a$. Let $\overline{d_t}$ be the $t^{th}$ component of $w$. Then $a\overline{d_t}=\overline{1}$ and hence $ad_t+bp^{\alpha_t}=1$ for some $b \in \mathbb{Z}$. Since $p|a$, we have $p|1$, which is not possible. So $z$ is not adjacent to $w$. So $\deg(z)=p^{\alpha_t}-1$ and consequently, $\delta(\mathcal{G}(H))\leq p^{\alpha_t}-1$.

Thus we conclude from the above inequality and \eqref{DegAbelianpEq} that
$$\delta(\mathcal{G}(H))= p^{\alpha_t}-1.$$
This completes the proof of the theorem.
\end{proof}

\begin{theorem}\label{DisAbelianp}
Let $G$ be an abelian $p$-group for some prime $p$. Let $\psi : G \to \mathbb{Z}_{p^{\alpha_1}} \times \mathbb{Z}_{p^{\alpha_2}} \times \ldots \times \mathbb{Z}_{p^{\alpha_r}}$ be an isomorphism and $\tau(G)=p^{\alpha_t}$. If $g \in G$ is such that all components of $\psi(g)$ are $\overline{0}$ except $t^{th}$, say $\overline{a}$, satisfying $\gcd(a,p)=1$, then $E[g,\psi^{-1}(\langle \psi(g) \rangle)-g]$ is a minimum disconnecting set of $\mathcal{G}(G)$.
\end{theorem}

\begin{proof}
Take $\psi(g)=z$. Following the proof of \Cref{DegAbelianp}, $N(z)=\langle z \rangle-z$. Then, $\psi$ being an isomorphism, $N(g)=\psi^{-1}(\langle z \rangle)-g$. Thus by \Cref{MinDisconSet}, $E[g,\psi^{-1}(\langle z \rangle)-g]$ is a minimum disconnecting set of $\mathcal{G}(G)$.
\end{proof}

\subsection{Dihedral groups}
\label{s-Dn}

For a positive integer $n \geq 3$, the \emph{dihedral group} $D_n$ \cite{Dummit} is a finite group of order $2n$ having presentation

\begin{equation}\label{DihedralEq}
D_n=\left \langle a,b \mid a^n=b^2=e, ab=ba^{-1} \right \rangle,
\end{equation}
where $e$ is the identity element of $D_n$.

 In the next theorem, we find the minimum degree and cut-edge of $\mathcal{G}(D_n)$.

\begin{theorem}\label{DihedralDegree}
For $n\geq 3$, $\delta(\mathcal{G}(D_n))=1$. Moreover, for any $ 0 \leq i < n$, edge between $e$ and $a^ib$ is a cut-edge of $\mathcal{G}(D_n)$.
\end{theorem}

\begin{proof}
From the presentation of $D_n$,  $\langle a \rangle=\{e, a, a^2,\ldots, a^{n-1}\}$. For any $ 0 \leq i < n$, $(a^ib)^2=e$, so that $\langle a^ib \rangle=\{e, a^ib\}$. Thus

 \begin{equation}
D_n=\langle a \rangle \cup \bigcup \limits_{i=0}^{n-1}\langle a^ib \rangle.
\end{equation}

 Then for any $0 \leq i < n$, the only vertex adjacent to $a^ib$ is $e$ and hence $\deg(a^ib)=1$. As $\mathcal{G}(D_n)$ is connected, $\deg(x) \geq 1$ for all $x \in D_n$. Hence $\delta(\mathcal{G}(D_n))=1$. Additionally, the edge between $e$ and $a^ib$ is a cut-edge of $\mathcal{G}(D_n)$ for all $0 \leq i < n$.
\end{proof}
\vspace{8pt}

\subsection{Dicyclic groups}
\label{s-Qn}

For a positive integer $n \geq 2$, the \emph{dicyclic group} $Q_n$ \cite{Dummit} is a finite group of order $4n$ having presentation
\begin{equation}\label{DicyclicEq}
Q_n=\left \langle a, b \mid a^{2n}=e, a^n=b^2, ab=ba^{-1} \right \rangle,
\end{equation}
where $e$ is the identity element of $Q_n$.

We first show by induction that $(a^ib)^2=a^n$ for all $0 \leq i \leq 2n-1$. As $b^2=a^n$, it is trivially true for $i=0$. Let it be true for $i=k$, where $0 \leq k \leq 2n-2$. Then for $i=k+1$, $(a^{k+1}b)^2=a^{k+1}ba^{k+1}b=a^{k}ba^{-1}a^{k+1}b=(a^kb)^2=a^n$, by induction hypothesis.

 Now for any $0 \leq i \leq n-1$, we have $(a^ib)^3=a^na^ib=a^{n+i}b$ and $(a^{n+i}b)^3=a^na^{n+i}b=a^ib$. Thus
\begin{equation}\label{EqQn}
\langle a^ib \rangle=\langle a^{n+i}b \rangle=\{ e, a^ib, a^n, a^{n+i}b \} \text{ for all } 0 \leq i \leq n-1.
\end{equation}

Therefore, any element of $Q_n - \langle a \rangle$ can be written as $a^ib$ for some $0 \leq i \leq 2n-1$. Subsequently, we have
\begin{equation}\label{EqQn2}
Q_n=\langle a \rangle \cup \bigcup \limits_{i=0}^{n-1}\langle a^ib \rangle
\end{equation}

In the next theorem, we find the minimum degree and minimum disconnecting sets of $\mathcal{G}(Q_n)$.

\begin{theorem}\label{DicyclicMinDeg}
For $n\geq 2$, $\delta(\mathcal{G}(Q_n))=3$. Moreover, for any $0 \leq i \leq n-1$, $E[a^ib, \{ e, a^n, a^{n+i}b \}]$ and $E[a^{n+i}b, \{ e, a^n, a^ib \}]$ are minimum disconnecting sets of $\mathcal{G}(D_n)$.
\end{theorem}

\begin{proof}
We follow the presentation of $Q_n$ in \eqref{DicyclicEq}. Let $x \in \langle a \rangle$. Since $o(a)=2n$, by \Cref{MinDegBasic}, $\deg(x)\geq \phi(2n)+1$. For $m > 2$, $\phi(m)$ is an even integer \cite[Theorem 7.4]{burton2006elementary}. So, in particular $\deg(x) \geq \phi(2n)+1\geq 3$.

Now let $y \in Q_n - \langle a \rangle$. As we have already observed, $y=a^ib$ for some $0 \leq i \leq 2n-1$. So from \eqref{EqQn2}, the only vertices adjacent to $y$ in $\mathcal{G}(Q_n)$ are elements of $\langle y \rangle-\{ y \}$, and from  \eqref{EqQn}, $\deg(y)=3$. Thus we conclude that $\delta(\mathcal{G}(Q_n))=3$.

Let $0 \leq i \leq n-1$. From the structure of $\mathcal{G}(Q_n)$, $N(a^ib) = \{ e, a^n, a^{n+i}b \}$ and $N(a^{n+i}b) = \{ e, a^n, a^ib \}$. Hence, by \Cref{MinDisconSet}, $E[a^ib, \{ e, a^n, a^{n+i}b \}]$ and $E[a^{n+i}b, \{ e, a^n, a^ib \}]$ are minimum disconnecting sets of $\mathcal{G}(D_n)$.
\end{proof}

\section{Equality of connectivity and minimum degree of power graphs}
\label{s-equal}

In this section, we investigate the equality of connectivity and minimum degree for power graphs of finite groups. We first discuss some necessary conditions required for the equality of connectivity and minimum degree of power graphs of finite groups (cf. \Cref{ConDegEq}), and find the minimum degree when the equality holds for cyclic groups (cf. \Cref{ConDegEqZn1}). We then supply a necessary and sufficient condition for the equality of connectivity and minimum degree of power graph of finite cyclic groups (cf. \Cref{ConDegEqZn2}). Subsequently, by using the minimum degrees obtained in \Cref{s-grps}, we address the equality for abelian $p$-groups, dihedral groups and dicyclic groups.

 We define a relation $\approx$ on $G$ as $x \approx y$ if $\langle x \rangle = \langle y \rangle$. Notice that it is an equivalence relation. For $x \in G$, the equivalence class of $x$ under $\approx$ is simply called the $\approx$-class of $x$ and is denoted by $[x]$. We need the following result on minimal separating sets of $\mathcal{G}(G)$.

\begin{lemma}[{\cite[Theorem 2.10]{ConPowerGr17}}]\label{MinSepUnion}
Let $G$ be a group and $T$ be a minimal separating set of $\mathcal{G}(G)$. Then for any $x \in G$, either $[x] \subseteq T$ or $[x] \cap T = \emptyset$.
\end{lemma}

 Let $G$ be a cyclic group of prime power order. Then it follows from \Cref{CompleteCond}(ii) that $\kappa(\mathcal{G}(G))=\kappa'(\mathcal{G}(G))=\delta(\mathcal{G}(G))=n-1$. The next theorem give some necessary conditions of the concerned equality for groups that are not cyclic groups of prime power order.

\begin{theorem}\label{ConDegEq}
Let $G$ be a finite group and $\kappa(\mathcal{G}(G))=\delta(\mathcal{G}(G))$. If $G$ is not a cyclic group of prime power order and $\delta(\mathcal{G}(G))=\deg(v)$, then the following hold:
\begin{enumerate}[\rm(i)]
\item $N(v)$ is a minimum separating set of  $\mathcal{G}(G)$.
\item The order of $v$ is $2$ in $G$. Consequently, $G$ is of even order.
\end{enumerate}
\end{theorem}

\begin{proof}
(i) Let $|G|=n$. By \Cref{CompleteCond}, $\mathcal{G}(G)$ is not a complete graph. Let $v \in G$ be the vertex such that $\delta(\mathcal{G}(G)) = \deg(v)$.

If possible let $N(v)=G-v$. Then $\delta(\mathcal{G}(G))=\deg(v)=n-1$, which implies $\deg(x)=n-1$ for all $x \in G$. This is not possible, as it would mean that $\mathcal{G}(G)$ is a complete graph. So $N(v) \neq G-v$, i.e. there exists at least one vertex $u$ non-adjacent to $v$ in $\mathcal{G}(G)$. Thus there does not exist any path from $u$ to $v$ in $\mathcal{G}(G)-N(v)$, and hence $N(v)$ is a separating set of $\mathcal{G}(G)$. Further, $|N(v)|=\delta(\mathcal{G}(G))=\kappa(\mathcal{G}(G))$. Thus we conclude that $N(v)$ is a minimum separating set of $\mathcal{G}(G)$.

\noindent
(ii) From the proof of (i), $N(v)$ is a minimal separating set of $\mathcal{G}(G)$. So it follows from \Cref{MinSepUnion} that either $[v] \subseteq N(v)$ or $[v] \cap N(v) = \emptyset$. However, $[v]-\{ v \} \subseteq N(v)$ and $v \notin N(v)$. This is possible only when $[v]-\{ v \}= \emptyset$, i.e. $\big|[v]\big|=1$. So, as $\big|[v]\big|=\phi(o(v))$, we have $o(v)=1$ or $o(v)=2$. If $o(v)=1$, then $v=e$, where $e$ is the identity element of $G$. But $N(e)=G-e$, which will in turn implies that $\mathcal{G}(G)$ is complete. Hence we have $o(v)=2$. Moreover, order of an element divides order of the group in a finite group, so that $|G|$ is even.
\end{proof}

\begin{lemma}\label{EulerSum}
Let $n \in \mathbb{N}$, $p$ be a prime factor of $n$ and $\alpha$ be the largest integer such that $p^\alpha | n$. Then for any integer $\beta$ satisfying $1 \leq \beta \leq \alpha$,
\begin{equation*}
\sum \limits_{d \bigm | \frac{n}{p^\beta}}\phi \left( \frac{n}{d} \right)=n - \frac{n}{p^{\alpha-\beta+1}}.
\end{equation*}

\end{lemma}

\begin{proof}
Taking $m=\dfrac{n}{p^\alpha}$, we have
\begin{align*}
 \sum \limits_{d|\frac{n}{p^\beta}} \phi \left( \frac{n}{d} \right) & =  \sum \limits_{d|m }\phi \left( \frac{n}{d} \right) +  \sum \limits_{d|m }\phi \left( \frac{n}{pd} \right) + \ldots +  \sum \limits_{d|m }\phi \left( \frac{n}{p^{\alpha-\beta} d} \right)\\
&=  \left \{ \phi(p^\alpha) + \phi(p^{\alpha-1}) + \ldots + \phi(p^\beta) \right \} \sum \limits_{d|m }\phi \left( \frac{m}{d} \right)\\
&= (p^\alpha - p^{\beta-1})m\\
& =n - \frac{n}{p^{\alpha-\beta+1}}
\end{align*}

\end{proof}

In the next theorem, when $\kappa(\mathcal{G}(\mathbb{Z}_n))$ and $\delta(\mathcal{G}(\mathbb{Z}_n))$ are equal, we find their common value, and also the  minimum separating set and minimum disconnecting set of  $\mathcal{G}(\mathbb{Z}_n)$.
\begin{theorem}\label{ConDegEqZn1}
If $n \in \mathbb{N}$ is not a prime power and $\kappa(\mathcal{G}(\mathbb{Z}_n))=\delta(\mathcal{G}(\mathbb{Z}_n))=k$ (say), then $k=deg \left ( \overline{\frac{n}{2}} \right )=n - \frac{n}{2^\alpha}$, where $\alpha$ is the largest integer such that $2^\alpha | n$.
\end{theorem}
\begin{proof}
Notice that there exists a vertex, say $\overline{a} \in \mathbb{Z}_n$, such that $\deg(\overline{a})=k$. Then by \Cref{ConDegEq}(ii), $o(\overline{a})=2$. Moreover, $o \left (\overline{\displaystyle \frac{n}{2}} \right)=2$, and if $d|n$, number of elements of order $d$ is $\phi(d)$ \cite[Theorem 4.4]{Gallian}. Since $\phi(2)=1$, $\overline{\displaystyle \frac{n}{2}}$ is the only element of order $2$ in $\mathbb{Z}_n$ and hence $\overline{a}=\overline{\displaystyle \frac{n}{2}}$. Furthermore, from \Cref{verdeg},
\begin{align*}
k = deg \left ( \overline{\frac{n}{2}} \right ) & = 1 + \sum \limits_{d|\frac{n}{2},d \neq \frac{n}{2}}\phi \left( \frac{n}{d} \right)\\
& = 1 + \sum \limits_{d|\frac{n}{2} }\phi \left( \frac{n}{d} \right)-\phi(2)\\
& = \sum \limits_{d|\frac{n}{2} }\phi \left( \frac{n}{d} \right)\\
& = n - \frac{n}{2^\alpha} \hspace{5pt} (\text{by } \Cref{EulerSum})
\end{align*}
\end{proof}

The following corollary is a consequence of \Cref{MinDisconSet}, \Cref{ConDegEq}(i) and \Cref{ConDegEqZn1}.

\begin{corollary}\label{MinSetZn}
Let $n \in \mathbb{N}$ and $n$ is not a prime power. If $\kappa(\mathcal{G}(\mathbb{Z}_n))=\delta(\mathcal{G}(\mathbb{Z}_n))$, then $\{ \overline{0} \} \cup \displaystyle \bigcup_{a \mid \frac{n}{2}, a \neq \frac{n}{2}} [\overline{a}]$, say $A$, is a minimum separating set and $E\left[\overline{\frac{n}{2}}, A\right]$ is a minimum disconnecting set of  $\mathcal{G}(\mathbb{Z}_n)$.
\end{corollary}

\begin{lemma}[{\cite[Theorem 2.27]{ConPowerGr17}}]\label{ConnValue}
If $n=p^\alpha q^\beta$, $p,q$ are distinct primes and $\alpha,\beta \in \mathbb{N}$, then $\kappa(\mathcal{G}(\mathbb{Z}_n)) = \phi(n)+p^{\alpha-1}q^{\beta-1}$.
\end{lemma}

We now obtain an equivalent condition for the equality of connectivity and minimum degree of power graph of $\mathcal{G}(\mathbb{Z}_n)$ in terms $n$.

\begin{theorem}\label{ConDegEqZn2}
For $n \in \mathbb{N}$, $\kappa(\mathcal{G}(\mathbb{Z}_n))=\delta(\mathcal{G}(\mathbb{Z}_n))$ if and only if $n=p^\alpha$ for some prime $p$ and $\alpha \in \mathbb{N}$ or $n=2q^\beta$ for some prime $q>2$ and $\beta \in \mathbb{N}$.
\end{theorem}

\begin{proof}
Let $\kappa(\mathcal{G}(\mathbb{Z}_n))=\delta(\mathcal{G}(\mathbb{Z}_n))$. By \Cref{CompleteCond}(ii), if $\mathcal{G}(\mathbb{Z}_n)$ is complete, then $n=p^\alpha$ for some prime $p$. Now suppose $\mathcal{G}(\mathbb{Z}_n)$ is not complete. So $n$ is not a prime power and hence by \Cref{ConDegEq}(ii), $n$ is even. Let $n=2^{\alpha_1}p_2^{\alpha_2}\ldots p_r^{\alpha_r}$, $r \geq 2$, $2 < p_2 < \cdots < p_r$ are primes and $\alpha_i \in \mathbb{N}$ for $1 \leq i \leq r$. Moreover, from \Cref{ConDegEqZn1}, $\delta(\mathcal{G}(\mathbb{Z}_n))=deg \left ( \overline{\displaystyle \frac{n}{2}} \right )$.

\begin{align*}
& \deg \left ( \overline{\frac{n}{2}} \right) -  \deg \left ( \overline{\frac{n}{2^{\alpha_1}}} \right)\\
& = 2 + \sum \limits_{d|\frac{n}{2}}\phi \left( \frac{n}{d} \right) - \left \{2^{\alpha_1} + \sum \limits_{d|\frac{n}{2^{\alpha_1}}}\phi \left( \frac{n}{d} \right) \right \} (\text{by } \Cref{verdeg}) \\
& = \left \{ 2 + n - \frac{n}{2^{\alpha_1}} - \phi(2) \right \}- \left \{ 2^{\alpha_1} + n -  \frac{n}{2} - \phi(2^{\alpha_1}) \right \} (\text{by } \Cref{EulerSum}) \\
& = \left ( 1- \frac{n}{2^{\alpha_1}}  \right )- \left (  2^{{\alpha_1}-1} - \frac{n}{2} \right ) \\
& = \left( \frac{n}{2} - \frac{n}{2^{\alpha_1}}  \right) - (2^{{\alpha_1}-1} -1) \\
& = \left( \frac{n}{2^{\alpha_1}}-1 \right) (2^{{\alpha_1}-1} -1)
\end{align*}
Since $r \geq 2$, we have $\displaystyle \frac{n}{2^{\alpha_1}} > 1$. So,  if ${\alpha_1} > 1$, then $2^{{\alpha_1}-1} >1$ and hence $\deg \left ( \overline{\displaystyle \frac{n}{2}} \right) > \deg \left ( \overline{\displaystyle \frac{n}{2^{\alpha_1}}} \right)$. This contradicts the fact that $\delta(\mathcal{G}(\mathbb{Z}_n))=deg \left ( \overline{\displaystyle \frac{n}{2}} \right )$. Thus $\alpha_1=1$, i.e. $n=2p_2^{\alpha_2}\ldots p_r^{\alpha_r}$.

Take $m=\displaystyle \frac{n}{2}$.  If possible let $r \geq 3$, so that $\displaystyle \frac{n}{2p_2^{\alpha_2}} \neq 1$.

\begin{align*}
&\deg \left ( \overline{\frac{n}{2}} \right) - \deg \left ( \overline{\frac{n}{2p_2^{\alpha_2}}} \right)\\
& = 2 +  \sum \limits_{d|\frac{n}{2},d \neq \frac{n}{2}}\phi \left( \frac{n}{d} \right) - \left \{ 2p_2^{\alpha_2} +  \sum \limits_{d|\frac{n}{2p_2^{\alpha_2}},d \neq \frac{n}{2p_2^{\alpha_2}}}\phi \left( \frac{n}{d} \right) \right \} (\text{by } \Cref{verdeg}) \\
& = 2 + \phi(2) \sum \limits_{d|m,d \neq m}\phi \left( \frac{m}{d} \right) - \left \{ 2p_2^{\alpha_2} + \phi(2) \sum \limits_{d|\frac{m}{p_2^{\alpha_2}},d \neq \frac{m}{p_2^{\alpha_2}}}\phi \left( \frac{m}{d} \right) \right \}\\
& = (2 + m -1)- \left \{ 2p_2^{\alpha_2} + m -  \frac{m}{p_2} - \phi(p_2^{\alpha_2}) \right \} (\text{by } \Cref{EulerSum}) \\
& = 1 + \frac{m}{p_2} - (p_2^{\alpha_2} + p_2^{\alpha_2-1}) \\
& = 1 + p_2^{\alpha_2-1} \left \{ \frac{m}{p_2^{\alpha_2}} - (p_2 + 1) \right \}
\end{align*}

Since $r \geq 3$, we have $\displaystyle \frac{m}{p_2^{\alpha_2}} \geq p_3 > p_2 + 1$. So $\deg \left ( \overline{\displaystyle \frac{n}{2}} \right) > \deg \left ( \overline{\displaystyle \frac{n}{2p_2^{\alpha_2}}} \right)$. This again contradicts the fact that $\delta(\mathcal{G}(\mathbb{Z}_n))=deg \left ( \overline{\displaystyle \frac{n}{2}} \right )$. Thus $r=2$, i.e. $n=2p_2^{\alpha_2}$. This completes the proof of forward implication.

We now prove the converse.  If $n=p^\alpha$ for some prime $p$ and $\alpha \in \mathbb{N}$, then by \Cref{CompleteCond}, $\mathcal{G}(\mathbb{Z}_n)$ is complete.  Thus $\kappa(\mathcal{G}(\mathbb{Z}_n))=\delta(\mathcal{G}(\mathbb{Z}_n))=n-1$. Now let $n=2q^\beta$ for some prime $q>2$ and $\beta \in \mathbb{N}$. Then by \Cref{MinDegValue}(i) and \Cref{ConnValue},  $\kappa(\mathcal{G}(\mathbb{Z}_n)) - \delta(\mathcal{G}(\mathbb{Z}_n)) =\phi(2q^\beta) + q^{\beta-1} - \{ 2+(2-1)(q^\beta-1)-1\} = q^\beta - q^\beta =0 $
\end{proof}

We now explore the relation between the connectivity and minimum degree of power graphs of some more groups.

\begin{theorem}
If $G$ be an abelian $p$-group, then $\kappa(\mathcal{G}(G))=\delta(\mathcal{G}(G))$ if and only if $\sigma(G)=1$ or $\tau(G)=2$.
\end{theorem}

\begin{proof}
If $G$ is cyclic, i.e. $\sigma(G)=1$ if and only if $\kappa(\mathcal{G}(G))=\delta(\mathcal{G}(G))=|G|-1$.

Now let $G$ be non-cyclic. By \Cref{DegAbelianp}, $\delta(\mathcal{G}(G))=\tau(G)-1$ and by \cite[Corollary 3.4]{ConPowerGr17}, $\kappa(\mathcal{G}(G))=1$. Thus $\kappa(\mathcal{G}(G))=\delta(\mathcal{G}(G))$ if and only if $\tau(G)=2$.
\end{proof}

\begin{theorem}\label{mec-other}
For $n \in \mathbb{N}$, the following hold:
\begin{enumerate}[\rm(i)]
\item For $n \geq 3$, $\kappa(\mathcal{G}(D_n))=\delta(\mathcal{G}(D_n))$.
\item For $n \geq 2$, $\kappa(\mathcal{G}(Q_n)) \neq \delta(\mathcal{G}(Q_n))$.
\end{enumerate}
\end{theorem}
\begin{proof}
It was shown in \cite{ChattopadhyayConnectivity} that $\kappa(\mathcal{G}(D_n)) = 1$ and $\kappa(\mathcal{G}(Q_n)) = 2$. On the other hand, by \Cref{DihedralDegree} we have $\delta(\mathcal{G}(D_n)) = 1$ and by \Cref{DicyclicMinDeg} we have $\delta(\mathcal{G}(Q_n)) = 3$. Hence the result follows.
\end{proof}

\section{Concluding remarks}
 We showed in \Cref{mindeg} that $\delta(\mathcal{G}(\mathbb{Z}_n)) \leq \eta_1(n)$ and $\delta(\mathcal{G}(\mathbb{Z}_n)) \leq \eta_2(n)$, and equality of these for certain classes of $n$ was shown in \Cref{MinDegValue}  beforehand. Thus, one may be interested in determining all $n \in \mathbb{N}$ for which $\eta_1(n)$ or  $\eta_2(n)$ is the minimum degree of $\mathcal{G}(\mathbb{Z}_n)$. Moreover, in \Cref{s-equal}, we obtained a necessary and sufficient condition for the equality of connectivity and minimum degree of $\mathcal{G}(\mathbb{Z}_n)$ and further examined it power graphs of some other groups. In the sequel, characterizing the groups for which the connectivity equals the minimum degree of their power graphs is open for study.


\begin{thebibliography}{10}

\bibitem{bubbo17}
D.~Bubboloni, M.~A. Iranmanesh, and S.~M. Shaker.
\newblock On some graphs associated with the finite alternating groups.
\newblock {\em Comm. Algebra}, 2017.
\newblock To appear.

\bibitem{burton2006elementary}
D.~M. Burton.
\newblock {\em Elementary number theory}.
\newblock Tata McGraw-Hill Education, 2006.

\bibitem{Cameron}
P.~J. Cameron.
\newblock The power graph of a finite group, {II}.
\newblock {\em J. Group Theory}, 13(6):779--783, 2010.

\bibitem{Ghosh}
P.~J. Cameron and S.~Ghosh.
\newblock The power graph of a finite group.
\newblock {\em Discrete Math.}, 311(13):1220--1222, 2011.

\bibitem{GhoshSensemigroups}
I.~Chakrabarty, S.~Ghosh, and M.~K. Sen.
\newblock Undirected power graphs of semigroups.
\newblock {\em Semigroup Forum}, 78(3):410--426, 2009.

\bibitem{chartrand1966graph}
G.~Chartrand.
\newblock A graph-theoretic approach to a communications problem.
\newblock {\em SIAM J. Appl. Math.}, 14:778--781, 1966.

\bibitem{CriticallyChartrand}
G.~Chartrand, A.~Kaugars, and D.~R. Lick.
\newblock Critically {$n$}-connected graphs.
\newblock {\em Proc. Amer. Math. Soc.}, 32:63--68, 1972.

\bibitem{ChattopadhyayConnectivity}
S.~Chattopadhyay and P.~Panigrahi.
\newblock Connectivity and planarity of power graphs of finite cyclic, dihedral
  and dicyclic groups.
\newblock {\em Algebra Discrete Math.}, 18(1):42--49, 2014.

\bibitem{chattopadhyay2015laplacian}
S.~Chattopadhyay and P.~Panigrahi.
\newblock On laplacian spectrum of power graphs of finite cyclic and dihedral
  groups.
\newblock {\em Linear and Multilinear Algebra}, 63(7):1345--1355, 2015.

\bibitem{curtin2014edge}
B.~Curtin and G.~Pourgholi.
\newblock Edge-maximality of power graphs of finite cyclic groups.
\newblock {\em Journal of Algebraic Combinatorics}, 40(2):313--330, 2014.

\bibitem{curtin2016euler}
B.~Curtin and G.~Pourgholi.
\newblock An euler totient sum inequality.
\newblock {\em Journal of Number Theory}, 2016.

\bibitem{doostabadi2015connectivity}
A.~Doostabadi and M.~Farrokhi D.~Ghouchan.
\newblock On the connectivity of proper power graphs of finite groups.
\newblock {\em Communications in Algebra}, 43(10):4305--4319, 2015.

\bibitem{Dummit}
D.~S. Dummit and R.~M. Foote.
\newblock {\em Abstract algebra}.
\newblock Wiley India, New Delhi, 2011.

\bibitem{MR3266285}
M.~Feng, X.~Ma, and K.~Wang.
\newblock The structure and metric dimension of the power graph of a finite
  group.
\newblock {\em European J. Combin.}, 43:82--97, 2015.

\bibitem{Gallian}
J.~A. Gallian.
\newblock {\em Contemporary abstract algebra}.
\newblock Cengage Learning India, New Delhi, 2013.

\bibitem{Halin-n-conn}
R.~Halin.
\newblock A theorem on {$n$}-connected graphs.
\newblock {\em J. Combinatorial Theory}, 7:150--154, 1969.

\bibitem{kelarev2000combinatorial}
A.~V. Kelarev and S.~J. Quinn.
\newblock A combinatorial property and power graphs of groups.
\newblock In {\em Contributions to General Algebra 12, Proceedings of the
  Vienna Conference}, pages 229--236, 2000.

\bibitem{kelarevDirectedSemigr}
A.~V. Kelarev and S.~J. Quinn.
\newblock Directed graphs and combinatorial properties of semigroups.
\newblock {\em Journal of Algebra}, 251(1):16--26, 2002.

\bibitem{MR3514980}
X.~Ma, M.~Feng, and K.~Wang.
\newblock The rainbow connection number of the power graph of a finite group.
\newblock {\em Graphs Combin.}, 32(4):1495--1504, 2016.

\bibitem{MR3200118}
A.~R. Moghaddamfar, S.~Rahbariyan, and W.~J. Shi.
\newblock Certain properties of the power graph associated with a finite group.
\newblock {\em J. Algebra Appl.}, 13(7):1450040, 18, 2014.

\bibitem{ConPowerGr17}
R.~P. Panda and K.~V. Krishna.
\newblock On connectedness of power graphs of finite groups.
\newblock 2017.
\newblock Submitted. arXiv:1703.08834.

\bibitem{plesnik1975critical}
J.~Plesn\'ik.
\newblock Critical graphs of given diameter.
\newblock {\em Acta Fac. Rerum Natur. Univ. Comenian. Math.}, 30:71--93, 1975.

\bibitem{MR3612206}
Y.~Shitov.
\newblock Coloring the {P}ower {G}raph of a {S}emigroup.
\newblock {\em Graphs Combin.}, 33(2):485--487, 2017.

\bibitem{whitney1932congruent}
H.~Whitney.
\newblock Congruent graphs and the connectivity of graphs.
\newblock {\em American Journal of Mathematics}, 54(1):150--168, 1932.

\end{thebibliography}

\end{document}